\documentclass[a4paper, reqno, 11pt]{amsart}

\usepackage[english]{babel}
\usepackage{amsmath}
\usepackage{amssymb}
\usepackage{amsthm}
\usepackage{mathrsfs}
\usepackage{enumerate}
\usepackage{ifthen}
\usepackage{bbm}
\usepackage{xcolor}
\usepackage{verbatim}

\usepackage[pdftex,     
plainpages=false,   
breaklinks=true,    
colorlinks=true,
linkcolor=blue,
citecolor=green,
pdftitle={Advection-diffusion equation with rough coefficients},
pdfauthor={P. Bonicatto, G. Ciampa and G. Crippa}
]{hyperref}

\usepackage{graphicx}
\usepackage{pgfplots}
\usepackage{tikz} 
\usepackage{tikz-3dplot}
\usetikzlibrary{patterns, patterns.meta}
\usepackage{float}
\usepackage{mathtools}
\usepackage[font=footnotesize,labelfont=bf]{caption}
\usepackage{xfrac}
\usepackage{bm}
\usepackage{bbm}
\renewcommand{\b}{\bm b}
\newcommand{\X}{\bm X}
\newcommand{\Xe}{\bm X^\e}
\newcommand{\E}{\mathbb{E}}

\newcommand{\margnote}[1]{
\ifthenelse{\boolean{shownotes}}%
{\marginpar{\raggedright\tiny\texttt{#1}}}%
{}%
}

\newcommand{\hole}[1]{
\ifthenelse{\boolean{shownotes}}%
{\begin{center} \fbox{ \rule {.25cm}{0cm}
\rule[-.1cm]{0cm}{.4cm} \parbox{.85\textwidth}{\begin{center}
\texttt{#1}\end{center}} \rule {.25cm}{0cm}}\end{center}}
{}
}
\newtheorem{thm}{Theorem}[section]
\newtheorem{prop}[thm]{Proposition}
\newtheorem{proposition}[thm]{Proposition}
\newtheorem{lem}[thm]{Lemma}
\newtheorem{cor}[thm]{Corollary}

\theoremstyle{definition}
\newtheorem{defn}[thm]{Definition}
\newtheorem{definition}[thm]{Definition}

\newtheorem*{mainthm*}{Selection Theorem}
\newtheorem*{mainthmdue*}{Regularity Theorem}

 \newtheorem{remark}[thm]{Remark}

\DeclareMathOperator{\dive}{\mathrm {div}}

\newcommand{\e}{\varepsilon}		       
\newcommand{\R}{\mathbb{R}}
\newcommand{\T}{\mathbb{T}}

\newcommand{\Z}{\mathbb{Z}}

\newcommand{\de}{\mathrm{d}}

\numberwithin{equation}{section}

\title[On the advection-diffusion equation with rough coefficients: vanishing viscosity]{On the advection-diffusion equation with rough coefficients: weak solutions and vanishing viscosity}

\author{Paolo Bonicatto}
\address[P.\ Bonicatto]{Mathematics Institute, University of Warwick,
	Zeeman Building, CV4 7HP Coventry, UK}
\email{paolo.bonicatto@warwick.ac.uk}

\author[Gennaro Ciampa]{Gennaro Ciampa}
\address[G.\ Ciampa]{Dipartimento di Matematica "Federigo Enriques", Universit\`a degli Studi di Milano, Via Cesare Saldini 50, 20133 Milano, Italy \& BCAM - Basque Center for Applied Mathematics, Mazarredo 14, E48009 Bilbao, Basque Country - Spain}
\email{gennaro.ciampa@unimi.it, gciampa@bcamath.org}

\author[Gianluca Crippa]{Gianluca Crippa}
\address[G. Crippa]{Department Mathematik Und Informatik, Universit\"at Basel, Spiegelgasse 1, CH-4051 Basel, Switzerland}
\email{gianluca.crippa@unibas.ch}

\date{\today}

\begin{document}

\vskip .2truecm

\begin{abstract}

	We deal with the vanishing viscosity scheme for the transport/continuity equation $\partial_t u + \dive (u\b) = 0$ drifted by a divergence-free vector field $\b$. Under general Sobolev assumptions on $\b$, we show the convergence of such scheme to the unique Lagrangian solution of the transport equation. 
	Our proof is based on the use of stochastic flows and yields quantitative rates of convergence. This offers a completely general selection criterion for the transport equation (even beyond the distributional regime) which compensates the wild non-uniqueness phenomenon for solutions with low integrability arising from convex integration constructions, as shown in recent works \cite{CL, MS3, MS, MS2}, and rules out the possibility of anomalous dissipation. 
	
	\footnotesize{
		}
\end{abstract}

\maketitle

\section{Introduction}

The goal of this paper is
to exploit the advection-diffusion equation 
\begin{equation*}
	\partial_t v + \dive(v\b) = \Delta v
\end{equation*}
to set up the \emph{vanishing viscosity scheme} for the linear transport/continuity equation
\begin{equation*}
	\partial_t u + \dive(u\b) = 0 
\end{equation*}
associated with a Sobolev, divergence-free vector field $\b$. This is carried out in a framework of low integrability for the solution that does not guarantee uniqueness at the level of transport/continuity equation. More precisely, for every $\e>0$ we consider the unique solution $v_\e$ to 
\begin{equation*}
	\partial_t v_\e + \dive(v_\e\b) = \e \Delta v_\e
\end{equation*}
and establish \emph{quantitative} convergence rates for $v_\e \to u^{\sf L}$ as $\e \downarrow 0$, where $u^{\sf L}$ is the Lagrangian solution of the transport equation, that is the solution transported by the flow of $\b$. Such a result fits into a well understood physical mechanism (the zero diffusivity/viscosity limit) and has also its own, mathematical interest: similar schemes have been proposed over the years for different equations (see e.g. for hyperbolic conservation laws \cite{BB, KR}, for fluid-dynamics \cite{masmoudi} and references therein). 

\subsection*{The transport equation drifted by Sobolev vector fields}

Given a vector field $\b \colon [0,T] \times \T^d \to \R^d$ on the $d$-dimensional torus $\T^d:= \R^d/\Z^d$, we consider the initial value problem for the transport/continuity equation associated with $\b$, i.e. 
\begin{equation}\label{eq:te-intro}\tag{TE}
	\begin{cases} 
		\partial_t u + \dive(u\b) = 0 \\
		u |_{t=0} = u_0, 
	\end{cases}
\end{equation}
for a given initial datum $u_0 \colon \T^d \to \R$. It is by now well known that \eqref{eq:te-intro} is deeply connected with the ordinary differential equation associated to $\b$ and more precisely with the \emph{regular Lagrangian flow} of $\b$ (see Definition \ref{def:rlf} below). This concept has proven to be the right generalization of the classical notion of flow in connection with such problems (see e.g. \cite{A04, BC}). If $\b \in L^1_t W^{1,p}_x$, existence and uniqueness of the regular Lagrangian flow $\bm X$ of $\b$ hold \cite{DPL}: in turn, a straightforward computation allows to check that the transport of the initial datum along the characteristics selected by the regular Lagrangian flow always defines a solution to \eqref{eq:te-intro}. More precisely, the function $u^{\mathsf L}(t,x) := u_0(\bm X^{-1}(t,\cdot)(x))$ is the unique distributional solution to \eqref{eq:te-intro}, whenever $u_0 \in L^q$ for some $q \ge 1$ with $\sfrac{1}{p} + \sfrac{1}{q} \le 1$, see \cite{DPL}. We will refer to $u^{\mathsf L}$ as the \emph{Lagrangian solution}. 

\subsubsection*{The need for a selection criterion for \eqref{eq:te-intro}}
A few years ago, in a series of innovative contributions, Modena and Sz\'ekelyhidi constructed a plethora of counterexamples, showing ill-posedness of the problem \eqref{eq:te-intro}. More precisely, in \cite{MS, MS2} the authors have produced divergence-free, Sobolev vector fields $\b \in C_t W^{1,p}_x$ such that \eqref{eq:te-intro} admits infinitely many distributional solutions $u \in C_t L^q_x$, with $\sfrac{1}{p}+\sfrac{1}{q}>1+\sigma(d)$ and $u\b \in L^1$. Here $\sigma(d)$ is a dimensional constant, which has been refined in \cite{MS3} to $\sigma(d) := \sfrac{1}{d}$.
Yet the situation in the intermediate regime $1< \sfrac{1}{p}+\sfrac{1}{q} \le 1+\sigma(d)$ is an open problem (see, however, \cite{CL}). Remarkably, in the works just mentioned, the authors build distributional solutions that do \emph{not} enjoy typical properties of smooth solutions, such as the conservation of the $L^p$ norms. It is therefore natural to ask whether 
such solutions can be obtained as limit of (physically or numerically) significant approximation procedures, as for instance the vanishing viscosity method.

\subsection*{Advection-diffusion equations and the vanishing viscosity scheme}

Given a vector field $\b \colon [0,T] \times \T^d \to \R^d$, one can consider the initial value problem for the \emph{advection-diffusion equation} associated with $\b$, i.e. 
\begin{equation}\label{eq:ad-intro}\tag{ADE}
	\begin{cases} 
		\partial_t v + \dive(v\b) = \Delta v \\
		v |_{t=0} = v_0, 
	\end{cases}
\end{equation}
where $v_0 \colon \T^d \to \R$ is a given initial datum. Due to the presence of the Laplacian, \eqref{eq:ad-intro} is a second-order parabolic partial differential equation. If the vector field $\b$ is smooth, classical existence and uniqueness results are available and can be found in standard PDEs textbooks (see e.g. \cite{EV}). The problem \eqref{eq:ad-intro} has been also studied outside the smooth framework in many classical references (besides quoting again \cite{EV}, we mention the monograph \cite{LA2} and the more recent book \cite{LBL_book}, which is intimately related to a fluid dynamics context). In particular, from \cite{LBL} we know that if $\b \in L^1_t W^{1,1}_x$ and $v_0 \in L^\infty$, then there exists a unique solution $v \in L_t^\infty L_x^2 \cap L_t^2H_x^1$ to \eqref{eq:ad-intro}. This allows one to set up the vanishing viscosity scheme and the main result of the paper is a \emph{quantitative version} of the following theorem:

\begin{mainthm*}[Vanishing viscosity for linear transport]
	Let $\b \in L^1_t W^{1,1}_x$ be a divergence-free vector field on the torus $\T^d$ and let $u_0 \in L^1$ be a given initial datum. Let $(v_0^\e)_\e \subset L^\infty_x $ be \emph{any} sequence of functions such that $v_0^\e \to u_0$ strongly in $L^1_x$.  Consider the \emph{vanishing viscosity solutions}, i.e. the sequence $(v_\e)_{\e>0} \subseteq L^\infty_t L^\infty_x \cap L_t^2 H_x^1$ of solutions to
	\begin{equation}\label{eq:vv-intro}\tag{VV}
		\begin{cases}
			\partial_t v_\e + \b \cdot \nabla v_\e = \e\Delta v_\e &  \text{ in } (0,T) \times \T^d \\
			v_\e\vert_{t=0}=v_0^\e&   \text{ in }  \T^d. 
		\end{cases} 
	\end{equation}
	Then $(v_\e)_{\e>0}$ converges in $C_tL^1_x$ to the Lagrangian solution $u^{\sf L}$ to \eqref{eq:te-intro}.  
\end{mainthm*}

For the precise statements of the quantitative results, we refer the reader to Section \ref{sec:lagrangian_proof} and \ref{sec:rates} below. We highlight the applicability range of such result, which is completely general: dealing with the (extreme) case of a $W^{1,1}$ field and an $L^1$ solution (in space), we prove that the vanishing viscosity scheme \emph{always} acts as a selection principle (even in an integrability regime where the product $u \b$ cannot be defined distributionally) and that the family $(v_\e)_\e$ always selects, in the limit, the Lagrangian solution $u^{\sf L}$.
Observe that, for the Lagrangian solution, all its $L^q$-norms are conserved (recall that we assume the vector field to be divergence-free). In particular, for $u_0 \in L^2$, our result rules out the possibility of anomalous dissipation in the vanishing viscosity limit, that is, it implies that 
\begin{equation*}
	\e \int_0^T \Vert \nabla v_\e \Vert^2_{L^2} \, \de t \to 0, \quad \text{as $\e \to 0$,}
\end{equation*}
see Remark \ref{rmk:no_anom_diss}. As in the case of the advection-diffusion equation \eqref{eq:ad-intro}, distributional solutions of \eqref{eq:te-intro} exist even if we require only integrability assumptions on $\b$. However, in contrast to the case of \eqref{eq:ad-intro}, for vector fields outside the Sobolev class there is a wide literature of counterexamples to the uniqueness for \eqref{eq:te-intro}, see for example \cite{Ai, ABC14, DPL, Dep}. There are many contexts in PDEs (conservation laws, fluid dynamics, etc...) in which the notion of distributional solution is too general to ensure uniqueness and therefore selection criteria are needed to characterize particularly meaningful solutions. The selection problem for the transport equation has already been posed in \cite{CCS2} (see also \cite{DLG}) where the authors considered as a (non) selection criterion the smooth approximation of the vector field. In particular, it is shown that a smooth approximation may produce different (Lagrangian) solutions in the limit if the vector field is not $W^{1,p}$ for any $p$. Moreover, results of Lagrangianity for weak solutions of the 2D Euler equations obtained via vanishing viscosity have been established in \cite{CNSS, CS}, see also \cite{CCS3} for the Lagrangianity of solutions obtained via vortex-blob approximations.

We present two proofs of the Selection Theorem. The first one, presented in Section \ref{sec:lagrangian_proof}, has a Lagrangian nature and is based on the use of stochastic flows (\cite{CJ,CCS4}). 
	Moreover, in Section \ref{sec:rates} we provide \emph{quantitative} rates of convergence of $v^\e$ to $u^{\sf L}$ and also quantitative (in the viscosity) stability estimates for solutions of the advection-diffusion equation. Such rates are compared with the ones obtained in the recent works \cite{BN} and \cite{S21}. The second proof (which is contained in the Appendix) is more Eulerian in spirit and is a slightly expanded version of the one contained in DiPerna-Lions' original contribution \cite{DPL}, based on a duality argument. We offer a comprehensive, detailed proof which ultimately reveals the complete generality of the vanishing viscosity scheme, which is able to bypass the distributional regime. In particular, we cover also the case $p=1$ which was left somehow implicit in \cite{DPL}.

\subsection*{Acknowledgements} Most of this work was developed while the first and second authors were Post-Doc at the Departement Mathematik und Informatik of the University of Basel supported by the ERC Starting Grant 676675 FLIRT. During the final preparation of the manuscript, P.B. has received funding from the European Research Council (ERC) under the European Union's Horizon 2020 research and innovation programme, grant agreement No 757254 (SINGULARITY). G. Ciampa has been supported by the Basque Government under program BCAM- BERC 2022-2025, by the Spanish Ministry of Science, Innovation and Universities under the BCAM Severo Ochoa accreditation SEV-2017-0718 and he is currently supported by the ERC Starting Grant 101039762 HamDyWWa. The authors acknowledge useful discussions with S. Spirito and S. Modena and they are grateful to M. Sorella for some comments on the first version of this work. Finally, the authors wish to thank the anonymous referee for useful comments and C. Le Bris for drawing the reference \cite{Lions_video} to their attention.

\section{Preliminaries and notations}
We begin by fixing the notation and recalling some basic facts we will need in the following. 

\subsection{Notation} Throughout the paper, $d \ge 1$ will be a fixed integer. We will denote by $\T^d := \R^d/\Z^d$ the $d$-dimensional flat torus and by $\mathscr L^d$ the Lebesgue measure on it. We identify the $d$-dimensional flat torus with the cube $[0,1)^d$ and we denote with $\mathsf{d}$ the geodesic distance on $\T^d$, which is given by
\begin{equation*}
	\mathsf{d}(x,y)=\min\{|x-y-k|:k\in\Z^d\,\,\mbox{such that }|k|\leq 2\}.
\end{equation*}
We will use the letters $p,q$ to denote real numbers in $[1,+\infty]$ and $p'$ will be the (H\"older) conjugate to $p$. We will adopt the customary notation for Lebesgue spaces $L^p(\T^d)$ and for Sobolev spaces $W^{k,p}(\T^d)$; in particular, $H^k(\T^d) := W^{k,2}(\T^d)$. We will denote with $\|\cdot\|_{L^p}$ (respectively $\|\cdot\|_{W^{k,p}}$,$\|\cdot\|_{H^k}$) the norms of the aforementioned functional spaces, omitting the domain dependence when not necessary. Every definition below can be adapted in a standard way to the case of spaces involving time, like e.g. $L^1([0,T];L^p(\T^d))$. 

\subsubsection*{Equi-integrability}

We recall the definition of equi-integrability for a family of functions in $L^1$:
\begin{defn}[Equi-integrability]\label{def:equii}
	A family $\{ \varphi_i\}_{i\in I}\subset L^1(\T^d)$ is {\em equi-integrable} if for every $\e>0$ there exists $\delta>0$ such that for every Borel set $E\subset\T^d$ with $\mathscr{L}^d(E)\leq\delta$ it holds 
	\begin{equation*}
		\int_E|\varphi_i|\de x\leq\e \qquad \text{ for every $i \in I$.}
	\end{equation*}
\end{defn}
The following well-known results offer useful criteria to check the equi-integrability of a family of functions in $L^1$:

\begin{thm}[Dunford-Pettis, de la Vall\'ee-Poussin]\label{thm:equiintegrability}
	Let $\{ \varphi_i\}_{i\in I}\subset L^1(\T^d)$ be a bounded family. Then the following are equivalent:
	\begin{enumerate}
		\item[\rm (i)] $\{ \varphi_i\}_{i\in I}$ is equi-integrable;
		\item[\rm (ii)] $\{ \varphi_i\}_{i\in I}$ is weakly sequentially pre-compact in $L^1(\T^d)$;
		\item[\rm (iii)] there exists a non-negative, increasing, convex function $\Psi \colon [0,+\infty) \to [0,+\infty)$ such that
		\begin{equation*}
			\lim_{t\to+\infty}\frac{\Psi(t)}{t}=+\infty \quad \text{ and } \quad \sup_{i\in I}\int_{\T^d}\Psi(|\varphi_i|)\de x<+\infty.
		\end{equation*}
	\end{enumerate}
\end{thm}

We finish this subsection with the following useful lemma.
\begin{lem}\label{lem:decomposition-equi}
	Let $\{ \varphi_i\}_{i\in I}\subset L^1(\T^d)$ be a bounded family. Then, $\{ \varphi_i\}_{i\in I}$ is equi-integrable if and only if for any $r\in[1,\infty]$ and $\e>0$ there exist $\{g_i^1\}_{i\in I}\subset L^1(\T^d)$, $\{ g_i^2 \}\subset L^r(\T^d)$, and a constant $C_\e>0$ such that
	\begin{equation*}
		f_i=g_i^1+g_i^2,\hspace{0.5cm}\sup_{i\in I}\|g_i^1\|_{L^1}\leq \e,\hspace{0.5cm}\sup_{i\in I}\|g_i^2\|_{L^r}\leq C_\e.
	\end{equation*}
\end{lem}
\begin{proof}
	Let $\{ \varphi_i\}_{i\in I}\subset L^1(\T^d)$ be an equi-integrable sequence such that $\sup_{i\in I}\|\varphi_i\|_{L^1}\leq C$. Let $\e>0$ be fixed and let $\delta>0$ as in Definition \ref{def:equii}. Then, we define the set
	$$
	A_\delta^i:=\left\{x\in\T^d: |\varphi_i(x)|>\frac{C}{\delta} \right\},
	$$
	and by Chebishev inequality we know that
	$$
	\sup_{i\in I}\mathscr{L}^d(A_\delta^i)\leq \frac{\delta}{C}\|\varphi_i\|_{L^1}\leq \delta.
	$$
	So, by the equi-integrability 
	$$
	\sup_{i\in I}\int_{A_\delta^i}|\varphi_i|\de x\leq \e, 
	$$
	and it is now clear that, by defining $g_i^1=\varphi_i\chi_{A_\delta^i}$ and $g_i^2=\varphi_i(1-\chi_{A_\delta^i})$, we have that
	$$
	\sup_{i\in I}\|g_i\|_{L^1}\leq \e,\hspace{0.5cm}\sup_{i\in I}\|g_i^2\|_{L^r}\leq C_\e,
	$$
	since $\T^d$ has finite measure.
	
	We now prove the opposite implication. Let $r\in[1,\infty]$ and $\e>0$ be fixed, we consider a decomposition such that
	$$
	\|g_i^1\|_{L^1}\leq \e/2, \hspace{0.5cm}\|g_i^2\|_{L^r}\leq C_\e.
	$$
	Let us check that we can take $\delta=\left(\frac{\e}{2C_\e}\right)^{r/(r-1)}$ in the definition of equi-integrability. Indeed, if $A\subset\T^d$ is such that $\mathscr{L}^d(A)\leq \delta$, we have that
	\begin{equation*}
		\int_A|\varphi_i|\de x\leq \int_A|g_i^1|\de x+\int_A|g_i^2|\de x \leq \|g_i^1\|_{L^1}+\|g_i^2\|_{L^r}\mathscr{L}^d(A)^{(r-1)/r} \leq \e/2+\e/2 = \e. \qedhere 
	\end{equation*}
\end{proof}

\subsubsection*{Some Harmonic Analysis tools}
We will need to work with weak Lebesgue spaces, denoted by $M^p(\T^d)$: for the sake of completeness, we recall here their definition and some useful lemmata. 
\begin{defn}
	Let $u:\T^d\to\R$ be a measurable function. For any $1\leq p<\infty$ we define
	\begin{equation*}
		|||u|||_{M^p}^p=\sup_{\lambda>0}\left\{\lambda^p\mathscr{L}^d\left(\{x\in\T^d: |u(x)|>\lambda\}\right)\right\},
	\end{equation*}
	and we define the weak Lebesgue space $M^p(\T^d)$ as the set of the functions $u:\T^d\to\R$ with $|||u|||_{M^p}<\infty$. By convention, for $p=\infty$ we set $M^\infty(\T^d)=L^\infty(\T^d)$.
\end{defn}
Note that $|||\cdot|||_{M^p}$ is not subadditive, hence it is \emph{not} a norm. As a consequence, $M^p(\T^d)$ is not a Banach space. Moreover, since for every $\lambda>0$
$$
\lambda^p\mathscr{L}^d\left(\{x\in\T^d: |u(x)|>\lambda\}\right)=\int_{|u|>\lambda}\lambda^p\de x\leq\int_{|u|>\lambda}|u(x)|^p\de x\leq \|u\|_{L^p}^p,
$$
we have the inclusion $L^p(\T^d)\subset M^p(\T^d)$ and in particular $|||u|||_{M^p}\leq \|u\|_{L^p}$.
The following lemma shows that we can interpolate the spaces $M^1$ and $M^p$, with $p>1$, obtaining a bound on the $L^1$ norm.
\begin{lem}\label{lem:interp}
	Let $u:\T^d\to[0,\infty)$ be a non-negative measurable function. Then for every $p\in(1,\infty)$ we have the interpolation estimate
	\begin{equation*}
		\|u\|_{L^1}\leq \frac{p}{p-1}|||u|||_{M^1}\left[1+\log\left( \frac{|||u|||_{M^p}}{|||u|||_{M^1}} \right)\right],
	\end{equation*}
	while for $p=\infty$ we have
	\begin{equation*}
		\|u\|_{L^1}\leq |||u|||_{M^1}\left[1+\log\left( \frac{\|u\|_{L^\infty}}{|||u|||_{M^1}} \right)\right].
	\end{equation*}
\end{lem}

We recall the definition of the Hardy-Littlewood maximal function.
\begin{defn}\label{def:maxfunc}
	Let $f\in L^1(\T^d)$, we define $Mf$ the maximal function of $f$ as
	$$
	Mf(x)=\sup_{r>0}\frac{1}{\mathscr{L}^d(B_r)}\int_{B_r(x)}|f(y)|\de y \hspace{1cm}\mbox{for every }x\in\T^d.
	$$
\end{defn}
The following estimates hold.
\begin{lem}\label{lem:maximal-function}
	For every $1<p\leq\infty$ we have the strong estimate
	$$
	\|Mf\|_{L^p}\leq C_{d,p}\|f\|_{L^p},
	$$
	while for $p=1$ only the weak estimate 
	$$
	|||Mf|||_{M^1}\leq C_d\|f\|_{L^1}
	$$
	holds. 
\end{lem}
Finally, we recall the following estimate on the different quotients of a $W^{1,1}$ function.
\begin{lem}\label{lem:diff-quotients}
	Let $f\in W^{1,1}(\T^d)$. Then there exists a negligible set $\mathcal{N}\subset\T^d$ such that
	\begin{equation*}
		|f(x)-f(y)|\leq C(d) \mathsf{d}(x,y)\left( M Df(x)+M Df(y) \right),
	\end{equation*}
	for every $x,y\in\T^d\setminus \mathcal{N}$, where $Du$ is the distributional derivative of $u$. 
\end{lem}

\section{The transport equation. Setup of the vanishing viscosity scheme}

\subsection{The transport equation and regular Lagrangian flows}
In the following, we will consider the initial value problem for the tranport/continuity equation 
\begin{equation}\label{eq:transport}
	\begin{cases}
		\partial_t u + \dive(\b u) = 0 & \qquad  \text{ in } (0,T) \times \T^d \\
		u\vert_{t=0}=u_0 & \qquad  \text{ in }  \T^d
	\end{cases}
\end{equation}
where $T>0$, $\bm b \colon [0,T] \times \T^d \to \R^d$ is a given divergence-free vector field and $u_0 \colon \T^d \to \R$ is the initial datum. We will work in Sobolev classes for the velocity field and the equation \eqref{eq:transport} will be understood in the sense of distributions. We explicitly observe that, since we are working on the torus, the integrability of $\b$ is sufficient to prevent the blow up of its trajectories and thus we can work with the \emph{regular Lagrangian flow} of $\b$: 

\begin{defn}[Regular Lagrangian flow]\label{def:rlf}
	Let $\b\in L^1((0,T);L^1(\T^d))$ be a divergence-free vector field. A map $\bm X \colon (0,T)\times(0,T)\times\T^d \to \T^d$ is a {\em regular Lagrangian flow} of $\b$ if the following conditions hold: 
	\begin{itemize}
		\item for a.e. $x\in \T^d$ and for any $t\in[0,T]$ the map $s\in[0,T]\mapsto \X(t,s,x)= \X_{t,s}(x)\in\T^d$ is an absolutely continuous solution of 
		\begin{equation*}
			\begin{cases}
				\partial_s \X_{t,s}(x)=\b(s,\X_{t,s}(x)) & s\in[0,T],\\
				\X_{t,t}(x)=x. 
			\end{cases}
		\end{equation*}
		\item For any $t\in[0,T]$ and $s\in[0,T]$ the map $x\in\T^d\mapsto \X_{t,s}(x)\in\T^d$ is measure-preserving.
	\end{itemize} 
\end{defn}
Existence and uniqueness of the regular Lagrangian flow of a Sobolev, divergence-free vector field $\b$ are ensured by \cite{DPL} and therefore we can give the following definition:
\begin{definition} Let $\b \in L^1((0,T);W^{1,1}(\T^d))$ be a divergence-free vector field and let $\bm X \colon (0,T) \times (0,T) \times \T^d \to \T^d$ be its regular Lagrangian flow. If $u_0 \in L^1(\T^d)$, then the map 
	\begin{equation*}
		u^{\mathsf L}(t,x) := u_0(\X_{t,0}(x))
	\end{equation*}
	is called \emph{Lagrangian solution} to \eqref{eq:transport}. 
\end{definition}

Observe that, under the assumption that $\b$ is divergence-free, if $u_0\b \in L^1(\T^d)$ then the Lagrangian solution is also a distributional solution to \eqref{eq:transport}.

\subsection{Setup of the vanishing viscosity scheme}

For each $\e>0$ we introduce the parabolic problem 
\begin{equation}\label{eq:VV_2}
	\begin{cases}
		\partial_t v_\e + \b \cdot \nabla v_\e = \e\Delta v_\e &  \text{ in } (0,T) \times \T^d \\
		v_\e\vert_{t=0}=v_0^\e&   \text{ in }  \T^d
	\end{cases} 
\end{equation}
being $v_0^\e$ a suitable bounded approximation of the initial datum $u_0$. We recall the following proposition.
\begin{proposition}[{\cite[Proposition 5.3]{LBL}}]\label{prop:parabolic_wp} Let $\bm b \in L^1((0,T); W^{1,1}(\T^d))$ be a divergence-free vector field and let $v_0 \in L^\infty(\T^d)$ be given. Then the problem 
	\begin{equation}\label{eq:VV}
		\begin{cases}
			\partial_t v + \dive(\b v) = \Delta v &  \text{ in } (0,T) \times \T^d \\
			v\vert_{t=0}=v_0 & \text{ in } \T^d 
		\end{cases}
	\end{equation}
	admits a unique 
	solution $v \in L^\infty((0,T); L^2(\T^d)) \cap L^2((0,T),H^1(\T^d))$. Furthermore, it holds $v \in L^\infty((0,T); L^\infty(\T^d))$ and 
	\begin{equation}\label{eq:enunciato_stima_for_all_s}
		\|v \|_{L^\infty((0,T); L^s(\T^d))} \le \|v_0 \|_{L^s(\T^d)}
	\end{equation}
	for any real number $s \in [1,+\infty]$.
\end{proposition}

In view of Proposition \ref{prop:parabolic_wp}, the problem \eqref{eq:VV_2} admits a unique solution $v^\e \in L^\infty((0,T); L^2(\T^d)) \cap L^2((0,T),H^1(\T^d))$, hence 
the family $(v_\e)_{\e>0}$ is well-defined. Our goal will be to establish (weak) compactness bounds on the family $(v_\e)_{\e>0}$ and characterize its limit points. We will show that a ``selection principle'' holds: the sequence $(v_\e)_\e$ always converges as $\e\to 0$ to the Lagrangian solution to \eqref{eq:transport}.	
	
\subsection{Stochastic flows}
We now introduce the {\em stochastic Lagrangian} formulation of the system \eqref{eq:VV_2}. Let $(\Omega, (\mathcal{F}_t)_{t\geq 0}, \mathbb{P})$ be a given filtered probability space, and let $\bm W_t$ be a $\T^d$-valued Brownian motion adapted to the backward filtration, i.e. for any fixed $t\in[0,T]$ and any $s\in[0,t]$, the Brownian motion $\bm W_{s}$ is such that $\bm W_t=0$. We have the following definition.
\begin{defn}[Stochastic flows]\label{def:sf} Let $\e>0$ and let $\b\in L^1((0,T);L^1(\T^d))$ be a divergence-free vector field. A map $\Xe\colon (0,T)\times(0,T)\times\T^d\times\Omega\to \T^d$ is a {\em stochastic flow} of $\b$ if 
	\begin{itemize}
		\item for any $t\in[0,T]$, for any $x\in \T^d$ and for a.e. $\omega\in\Omega$, the map $s \in [0,t] \mapsto \Xe(t,s,x,\omega) = \Xe_{t,s}(x,\omega)\in\T^d$ is a continuous solution to
		\begin{equation}
			\begin{cases}
				\de \Xe_{t,s}(x,\omega)=\b(s,\Xe_{t,s}(x,\omega))\de s + \sqrt{2\e}\, \de \bm W_s(\omega),\hspace{0.5cm}s\in[0,t),\\
				\Xe_{t,t}=x,
			\end{cases}\label{eq:sde}
		\end{equation}
		\item for any $t\in[0,T]$ and $s\in[0,t]$ and a.e. $\omega\in\Omega$ the map $x \in \T^d \mapsto \Xe_{t,s}(x,\omega)\in\T^d$ is measure preserving.
	\end{itemize}
\end{defn}
The celebrated Feynman-Kac formula, see \cite{K}, gives an explicit representation of the solution $v_\e$ of \eqref{eq:VV_2} in terms of the stochastic flow of $\b$, that is
\begin{equation*}
	v_\e(t,x)=\mathbb{E}[v_0^\e(\Xe_{t,0}(x))],
\end{equation*}
where we have used the standard notation $\mathbb{E}[f]$ to denote the average with respect to $\mathbb{P}$, that is
\begin{equation*}
	\E[f]:=\int_\Omega f(\omega)\de \mathbb{P}(\omega).
\end{equation*}
We remark that by considering a divergence-free vector field $\b\in L^1((0,T); W^{1,1}(\T^d))$ we have \emph{strong} existence and pathwise uniqueness for \eqref{eq:sde}: this means that we can construct a solution $\Xe$ to \eqref{eq:sde} on any given filtered probability space equipped with any given adapted Brownian motion, see \cite{CJ}. We remark that, since we are working on the torus, the boundedness assumption in \cite{CJ} can be dropped.

\section{A Lagrangian approach to the vanishing viscosity scheme}
\label{sec:lagrangian_proof}

In this section, we aim at giving a proof (exploiting Lagrangian tecnhiques) of the convergence of the vanishing viscosity scheme. In order to do that, we first establish some stability estimates between the stochastic and the deterministic flows.

\begin{lem}\label{lem:stability-flows}
	Let $\b\in L^1((0,T);W^{1,p}(\T^d))$ be a divergence-free vector field, where $p\geq 1$. Let $\X, \Xe$ be, respectively, the regular Lagrangian flow and the stochastic flow of $\b$. Then, 
	\begin{itemize}
		\item[$(i)$] if $p=1$ and $\b\in L^q((0,T)\times\T^d)$ for some $q>1$, then for every $\gamma>0$ there exists a constant $C_\gamma$ such that for a.e. $t\in [0,T]$ and $s\in [0,t]$
		\begin{equation}
			\label{est:p=1}
			\int_{\T^d}\mathbb{E}[\mathsf{d}(\Xe_{t,s}(x),\X_{t,s}(x))]\de x\leq C(T,p)\left(\sqrt[4]{\e}+\frac{C_\gamma}{|\ln\e|}+\frac{1}{|\ln\sqrt{\e}|}\gamma\left[ 1+\ln^+ \left(\frac{\|\b\|_{L^q}}{\sqrt{\e}\gamma}\right) \right]\right).
		\end{equation}
		\item[$(ii)$] If $p>1$, there exists a constant $C(T,p)$ such that for a.e. $t\in [0,T]$ and $s\in [0,t]$
		\begin{equation}
			\label{est:p>1}
			\int_{\T^d}\mathbb{E}[\mathsf{d}(\Xe_{t,s}(x),\X_{t,s}(x))]\de x\leq C(T,p)\left(\sqrt[4]{\e}+\frac{\|\nabla \b\|_{L^1L^p}}{|\ln\e|}  \right).
		\end{equation}
	\end{itemize}
	Moreover, the estimates \eqref{est:p=1}, \eqref{est:p>1} give the $L^1$-convengence of $\Xe_{t,s}$ towards $\X_{t,s}$ as $\e\to 0$.
\end{lem}
\begin{proof}
	We divide the proof in several steps.
	
	{\bf \emph{Step 1. The case $p=1$. } }For any $t\in(0,T)$, a.e. $\omega\in\Omega$ and a.e. $x\in\T^d$, the difference of the flows $\Xe-\X$ satisfies the following S.D.E. for $s\in[0,t]$
	\begin{equation}\label{eq:difference}
		\begin{cases}
			\de(\Xe_{t,s}(x,\omega)-\X_{t,s}(x))=(\b(s, \Xe_{t,s}(x,\omega))-\b(s,\X_{t,s}(x)))\,\de s+\sqrt{2\e}\de \bm W_{s}(\omega),\\
			\Xe_{t,t}(x,\omega)-\X_{t,t}(x)=0.
		\end{cases}
	\end{equation}
	We define the function the function $q_{\delta}(y)=\ln\left(1+\frac{|y|^2}{\delta^2}\right)$ and the related functional $Q^\delta_\e$ as
	\begin{equation}\label{def:Qdelta}
		Q^\delta_\e(t,s,x,\omega):=q_{\delta}(\Xe_{t,s}(x,\omega)-\X_{t,s}(x))=\ln\left(1+\frac{|\Xe_{t,s}(x,\omega)-\X_{t,s}(x)|^2}{\delta^2}\right),
	\end{equation}
	where $\delta>0$ is a fixed parameter that will be chosen later. An application of It\^{o}'s formula gives that
	\begin{align*}
		\int_{\T^d} \E\left[Q^\delta_\e(t,s,x)\right]\de x&=\int_s^t\int_{\T^d} \E\left[\nabla_{y} q_\delta(t,\tau,\Xe_{t,\tau}(x)-\X_{t,\tau}(x))\cdot\left(\b(\tau,\Xe_{t,\tau}(x))-\b(\tau,\X_{t,\tau}(x)) \right)\right]\de x\de \tau\\ 
		&+\e\int_s^t\int_{\T^d} \E\left[\nabla_{y}^2 q_\delta(t,\tau,\Xe_{t,\tau}(x)-\X_{t,\tau}(x))\right]\de x\de \tau,
	\end{align*}
	and from the inequalities
	$$
	\left| \nabla \ln\left(1+\frac{|y|^2}{\delta^2} \right)\right|\leq\frac{C}{\delta+|y|}, \hspace{0.7cm}\left| \nabla^2 \ln\left(1+\frac{|y|^2}{\delta^2}\right)\right|\leq\frac{C}{\delta^2+|y|^2},
	$$
	we obtain the following bound
	\begin{equation}\label{est:q}
		\int_{\T^d} \E\left[Q^\delta_\e(t,s,x)\right]\de x\leq \frac{\e (t-s)}{\delta^2}+C\int_s^t\int_{\T^d} \E\left[\frac{\left| \b(\tau,\Xe_{t,\tau}(x))-\b(\tau,\X_{t,\tau}(x)) \right|}{\delta+\left| \Xe_{t,\tau}(x)-\X_{t,\tau}(x)\right|}\right]\de x\de \tau.
	\end{equation}
	We now use the characterization of the equi-integrability as in Lemma \ref{lem:decomposition-equi}. We fix $r>1$ and let $\gamma>0$ a parameter that will be chosen later. Then, using Lemma \ref{lem:decomposition-equi} we decompose $\nabla\b$ as 
	$$
	|\nabla \b|=g_1^\gamma+g_2^\gamma,
	$$
	with
	\begin{equation*}
		\|g_1^\gamma\|_{L^1}\leq \gamma,\hspace{0.5cm}\|g_2^\gamma\|_{L^r}\leq C_\gamma,
	\end{equation*}  
	where the constant $C_\gamma$ is increasing as $\gamma\to 0$. Finally, we introduce the function
	\begin{equation*}
		\varphi(t,s,x,\omega):=\min\left\{ \frac{|\b(s,\Xe_{t,s}(x,\omega))|+|\b(s,\X_{t,s}(x))|}{\delta}; g_1^\gamma(s,\Xe_{t,s}(x,\omega))+ g_1^\gamma(s,\X_{t,s}(x))\right\}.
	\end{equation*}
	Going back to \eqref{est:q}, using the definition of $\varphi$, we get that
	\begin{align*}
		& \int_s^t\int_{\T^d}\E\left[\frac{\left| \b(\tau,\Xe_{t,\tau}(x))-\b(\tau,\X_{t,\tau}(x)) \right|}{\delta+\left| \Xe_{t,\tau}(x)-\X_{t,\tau}(x)\right|}\right]\de x\de \tau \\ 
		\leq & \int_s^t\int_{\T^d} \E\left[ \varphi(t,\tau,x)\right] \de x \de \tau \\
		& +\int_s^t\int_{\T^d} \E \left[g_2^\gamma(\tau,\Xe_{t,\tau}(x))+ g_2^\gamma(\tau,\X_{t,\tau}(x))\right] \de x \de \tau.
	\end{align*}
	Since $g_2^\gamma\in L^r((0,T)\times\T^d)$, by Holder inequality we have that
	\begin{equation}\label{g2}
		\int_s^t\int_{\T^d}\E\left[g_2^\gamma(\tau,\Xe_{t,\tau}(x))+ g_2^\gamma(\tau,\X_{t,\tau}(x))\right] \de x \de \tau\leq 2T^{(r-1)/r}C_\gamma.
	\end{equation}
	We now want to apply the interpolation inequality of Lemma \ref{lem:interp} to $\varphi$: first, by using the measure preserving property of $\X$ and $\Xe$, we have that
	\begin{equation}\label{est:phi1}
		\|\varphi\|_{L^q}\leq \frac{C}{\delta}\|\b\|_{L^q}.
	\end{equation}
	Second, by Chebishev inequality
	\begin{equation}\label{est:phi2}
		||| \varphi |||_{M^1((0,T)\times(0,T)\times\T^d\times\Omega)}\leq C |||g_1^\gamma|||_{M^1((0,T)\times \T^d)}\leq C \|g_1^\gamma\|_{L^1((0,T)\times \T^d)}.
	\end{equation}
	We apply Lemma \ref{lem:interp} to $\varphi$. The fact that the function $z\in[0,\infty)\mapsto z\left[1+\ln^+\left(\frac{C}{z}\right)\right]\in[0,\infty)$ is non-decreasing (where $\ln^+(w) := \max\{0,\ln(w)\}$ for every $w \ge 0$) and the bounds \eqref{est:phi1} and \eqref{est:phi2} give
	\begin{equation}\label{phi}
		\|\varphi\|_{L^1((0,T)\times(0,T)\times\T^d\times\Omega)}\leq C\frac{q}{q-1} \|g_1^\gamma\|_{L^1}\left[ 1+\ln^+\left(\frac{\|\b\|_{L^q}}{\|g_1^\gamma\|_{L^1}}\frac{T^{1-\frac{1}{q}}}{\delta}\right) \right].
	\end{equation}
	Substituting \eqref{g2} and \eqref{phi} in \eqref{est:q} we finally obtain
	\begin{equation*}
		\int_{\T^d} \E\left[Q^\delta_\e(t,s,x)\right]\de x\leq \frac{\e T}{\delta^2}+2T^{(r-1)/r}C_\gamma+\frac{Cq}{q-1}\gamma\left[ 1+\ln^+\left(\frac{\|\b\|_{L^q}T^{1-\frac{1}{q}}}{\delta\gamma}\right) \right].
	\end{equation*}
	Next, by defining 
	\begin{equation}\label{def:Adelta}
		A_\delta(t,s):=\left\{(x,\omega)\in \T^d\times\Omega:\mathsf{d}(\Xe_{t,s}(x,\omega),\X_{t,s}(x))> \sqrt{\delta}\right\},
	\end{equation}
	we obtain that
	\begin{align}
		\label{est:stab}
		\sup_{t,s\in(0,T)}\left(\mathscr{L}^d\otimes\mathbb{P}\right)\left(A_\delta(t,s)\right)&\leq \frac{C}{|\ln\delta|}\int_{\T^d} \E\left[\ln\left(1+\frac{(\mathsf{d}(\Xe_{t,s}(x),\X_{t,s}(x)))^2}{\delta^2}\right)\right]\de x\\
		\nonumber&\leq \frac{C}{|\ln\delta|}\int_{\T^d} \E\left[Q^\delta_\e(t,s)\right]\de x\\
		\nonumber&\leq C(T,q,r)\left( \frac{\e }{\delta^2|\ln\delta|}+\frac{C_\gamma}{|\ln\delta|}+\frac{1}{|\ln\delta|}\gamma\left[ 1+\ln^+\left(\frac{\|\b\|_{L^q}}{\delta \gamma}\right) \right]\right),
	\end{align}
	where we have used that $\mathsf{d}(x,y)\leq |x-y|$ for any $x,y\in\T^d$. 
	Therefore, 
	\begin{equation} \label{eq:split}
		\begin{split} 
			\int_{\T^d}\mathbb{E}[\mathsf{d}(\Xe_{t,s}(x),\X_{t,s}(x))]\de x\ = &\int_{(\T^d\times\Omega)\setminus A_\delta(t,s)}\mathsf{d}(\Xe_{t,s}(x,\omega),\X_{t,s}(x))\de \mathbb{P}(\omega)\de x\\
			& +  \int_{A_\delta(t,s)}\mathsf{d}(\Xe_{t,s}(x,\omega),\X_{t,s}(x))\de \mathbb{P}(\omega)\de x\\
			\leq &  \sqrt{\delta}+\left(\mathscr{L}^d\otimes\mathbb{P}\right)\left(A_\delta(t,s)\right)
		\end{split}
	\end{equation}
	where we have used that $\mathscr{L}^d \otimes\mathbb{P}$ is a probability measure on $\T^d\times\Omega$ and the distance $\mathsf{d}$ on the torus is bounded. 
	Finally, we choose $\delta=\sqrt{\e}$ and plugging \eqref{est:stab} in \eqref{eq:split}, we get that
	\begin{equation*}
		\label{eq:exp_dist}
		\int_{\T^d}\mathbb{E}[\mathsf{d}(\Xe_{t,s}(x),\X_{t,s}(x))]\de x\leq C(T,q,r)\left(\sqrt[4]{\e}+\frac{C_\gamma}{|\ln\e|}+\frac{1}{|\ln\sqrt{\e}|}\gamma\left[ 1+\ln^+ \left(\frac{\|\b\|_{L^q}}{\sqrt{\e}\gamma}\right) \right]\right),
	\end{equation*}
	and this concludes the proof of the estimate \eqref{est:p=1}.
	
	{ \bf \emph{Step 2. The case $p>1$.}} The proof easily follows from the arguments of Step 1. Since $\nabla\b(t,\cdot)\in L^p(\T^d)$ for a.e. $t\in(0,T)$, we apply Lemma \ref{lem:decomposition-equi} pointwise in time choosing $r=p$, $g_1^\gamma=0$, $\gamma=0$, $g_2^\gamma(t,\cdot)=|\nabla\b(t,\cdot)|$ and $C_\gamma(t)=\|\nabla\b(t,\cdot)\|_{L^p}$. In particular, note that the bound in \eqref{g2} changes into
	\begin{equation*}
		\int_s^t\int_{\T^d}\E\left[g_2^\gamma(\tau,\Xe_{t,\tau}(x))+ g_2^\gamma(\tau,\X_{t,\tau}(x))\right] \de x \de \tau\leq 2\|C_\gamma\|_{L^1}=2\|\nabla\b\|_{L^1L^p},
	\end{equation*}
	and by substituting in \eqref{est:p=1} we get \eqref{est:p>1}.
	
	{\bf \emph{Step 3. Convergence of the flows.}}
	We now prove the convergence of $X^\e_{t,s}$ towards $X_{t,s}$ as $\e\to 0$. If $p>1$ it follows directly by \eqref{est:p>1} by letting $\e\to 0$. Then we analyze the case $(i)$: the strategy is to choose properly the parameter $\gamma$ in \eqref{est:p=1} independently from $\e$. In this regards, note that the last term on the right-hand side of \eqref{est:p=1} is uniformly bounded in $\gamma$ for $\e$ small and converges to $0$ as $\gamma\to 0$. Hence, for any given $\eta>0$ there exists $\gamma_0$ independent from $\e$ such that for all $\gamma\leq \gamma_0$
	\begin{equation*}
		\frac{C(T,q,r)}{|\ln\sqrt{\e}|}\gamma\left[ 1+\ln^+ \left(\frac{\|\b\|_{L^q}}{\sqrt{\e}\gamma}\right)\right]< \frac{\eta}{3}.
	\end{equation*}
	Now that the constant $\gamma$ is fixed, and so is $C_\gamma$, we can infer that there exists $\e_0(M)>0$ such that for all $\e\leq \e_0(\gamma)$
	\begin{equation*}
		C(T,q)\left(\sqrt[4]{\e}+\frac{C_\gamma}{|\ln\e|}\right)<\frac{2}{3}\eta,
	\end{equation*}
	and this concludes the proof of the convergence of the flows.
\end{proof}

The convergence result for $\e\to 0$ to the Lagrangian solution reads as follows:  

\begin{thm}\label{thm:main2} Let $u_0 \in L^1(\T^d)$ be a given initial datum and let $\b$ be a divergence-free vector field such that 
\begin{itemize}
	\item either $\b \in L^1((0,T);W^{1,p}(\T^d))$ for some $p>1$, 
	\item or $\b \in L^1((0,T); W^{1,1}(\T^d))\cap L^q((0,T)\times\T^d)$ for some $q>1$.
\end{itemize}
Let $(v_0^\e)_\e \subset L^\infty(\T^d)$ be \emph{any} sequence of functions such that $v_0^\e \to u_0$ strongly in $L^1(\T^d)$. Then the sequence $(v_\e)_{\e>0} \subset  L^\infty((0,T); L^\infty(\T^d)) \cap L^2((0,T); H^1(\T^d))$ of solutions to \eqref{eq:VV_2} converges in $C([0,T];L^1(\T^d))$ to the (unique) Lagrangian solution to \eqref{eq:transport}. 
\end{thm}
\begin{proof}
	First of all, as already observed, from Proposition \ref{prop:parabolic_wp} we deduce that for every fixed $\e>0$ there exists a unique function $v_\e \in L^\infty((0,T); L^\infty(\T^d)) \cap L^2((0,T); H^1(\T^d))$ solving \eqref{eq:VV_2}. Moreover, by the Feynman-Kac formula we know that $v_\e$ satisfies
	$$
	v_\e(t,x)=\mathbb{E}[v_0^\e(\Xe_{t,0}(x))].
	$$
	On the other hand, the Lagrangian solution to \eqref{eq:transport} is given by
	$$
	u^{\mathsf L}(t,x)=u_0(\X_{t,0}(x)).
	$$
	Having both $v_\e$ and $u^{\mathsf L}$ a representation formula in terms of the flow, we use the stability of the flows to prove the convergence in the inviscid limit. We consider a sequence $u_0^n$ of Lipschitz approximations of $u_0$, then for any $t\in(0,T)$ we have that
	\begin{equation*}
		\begin{aligned}
			\|v_\e(t,\cdot)-u^{\mathsf L}(t,\cdot)\|_{L^1}&=\|\mathbb{E}[v_0^\e(\Xe_{t,0})]-u_0(\X_{t,0})\|_{L^1}\\
			&\leq\int_{\T^d}\int_\Omega |v_0^\e(\Xe_{t,0}(x,\omega))-u_0(\Xe_{t,0}(x,\omega))|\de\mathbb{P}(\omega)\de x\\
			&+\int_{\T^d}\int_\Omega |u_0^n(\Xe_{t,0}(x,\omega))-u_0(\Xe_{t,0}(x,\omega))|\de\mathbb{P}(\omega)\de x\\
			&+\int_{\T^d}|u_0^n(\X_{t,0}(x))-u_0(\X_{t,0}(x))|\de x\\
			&+\int_{\T^d}\int_\Omega |u_0^n(\Xe_{t,0}(x,\omega))-u_0^n(\X_{t,0}(x))|\de\mathbb{P}(\omega)\de x.
		\end{aligned}
	\end{equation*}
	In particular, by using that $u_0^n$ is Lipschitz and the measure preserving property of the flows, we get that
	\begin{equation}\label{est:final-main2}
		\|v_\e(t,\cdot)-u^{\mathsf L}(t,\cdot)\|_{L^1}\leq\|v_0^\e-u_0\|_{L^1}+2\|u_0^n-u_0\|_{L^1}+C_n\|\mathbb{E}[\mathsf{d}(\Xe_{t,0},\X_{t,0})]\|_{L^1}. 
	\end{equation}
	We first fix $n$ big enough, independently from $t$ and $\e$, in order to make the second term in \eqref{est:final-main2} as small as we want. Then the conclusion follows from Lemma \ref{lem:stability-flows}.
\end{proof}

	\begin{remark}[No anomalous dissipation]\label{rmk:no_anom_diss}
		If $u_0 \in L^2(\T^d)$,
		Theorem \ref{thm:main2} gives that $(v_\e)_{\e}$ converges in $C([0,T]; L^2(\T^d))$ to the Lagrangian solution of \eqref{eq:transport}. 
		In particular, from the identity 
		\begin{equation*}
			\frac{1}{2}\Vert v_\e(t,\cdot) \Vert_{L^2}^2 + \e \int_0^t \Vert \nabla v_\e(s,\cdot)\Vert_{L^2}^2\de s = 	\frac{1}{2}\Vert v_\e(0,\cdot) \Vert_{L^2}^2
		\end{equation*}
		valid for every $\e>0$ we deduce that 
		\begin{equation}\label{eq:no_an_diss}
			\e \int_0^t \Vert \nabla v_\e(s,\cdot)\Vert_{L^2}^2\de s \to 0, \quad \text{as $\e \to 0$}
		\end{equation}
		for every $t>0$. This means that \emph{no anomalous dissipation is possible} for the vanishing viscosity limit in the case either $\b \in L^1((0,T);W^{1,p}(\T^d))$ for some $p>1$, or $\b \in L^1((0,T);W^{1,1}(\T^d)) \cap L^q((0,T)\times \T^d)$ for some $q>1$ and $u \in L^\infty((0,T); L^2(\T^d))$, even though the solution lacks the integrability required for the DiPerna-Lions' theory to apply. We also remark that, using the results contained in Section \ref{sec:rates}, one can provide a quantitative rate for the convergence \eqref{eq:no_an_diss}.
		
		In a similar spirit, if $u_0 \in L^q(\T^d)$, we obtain that 
		\begin{equation*}
			\Vert u(t,\cdot) \Vert_{L^q(\T^d)} = \Vert u_0 \Vert_{L^q(\T^d)} \qquad \text{for every $t>0$,}
		\end{equation*}
		and more generally all Casimirs of the solution obtained as vanishing viscosity limit are conserved, that is for every $f$ it holds 
		\begin{equation*}
			\int_{\T^d} f(u(t,x))\, \de x = \int_{\T^d} f(u_0(x))\, \de x \qquad \text{for every $t>0$}.
		\end{equation*}
		
		On the other hand, vector fields in the class $L^1((0,T);C^\alpha(\T^d))$, with $d\geq 2$ and $\alpha\in[0,1)$, may exhibit anomalous dissipation as shown in \cite{DEIJ}.
	\end{remark}

	\begin{remark} 
		Comparing the statements of the Selection Theorem given in the Introduction with Theorem \ref{thm:main2}, one can observe that in the latter we need to make, in the case $\b \in L^1((0,T);W^{1,1}(\T^d))$, the additional integrability assumption $\b\in L^q((0,T)\times\T^d)$ for some $q>1$. 
		This is necessary to apply Lemma \ref{lem:stability-flows} which, in turn, yields quantitative convergence rates, as it will be shown in the next section. However, if one dispenses with quantitative estimates, 
		this additional assumption can be dropped and the Selection Theorem holds as stated in the Introduction. We refer the reader to the Appendix \ref{appendix} for the qualitative proof (without the additional technical integrability assumption), which is a minor refinement of the original argument of \cite{DPL}. 
\end{remark}

\section{Rates of convergence}\label{sec:rates} 
The goal of this section is to show that Lagrangian techniques are particularly useful in order to obtain explicit rates of convergence for the vanishing viscosity limit. To find such rates, we need slightly stronger integrability/regularity assumptions on the data. The first result deals with bounded initial data.
\begin{prop}\label{thm:rate}
	Let $\b\in L^1((0,T);W^{1,p}(\T^d))$ be a divergence-free vector field with $p>1$ and $u_0\in L^\infty(\T^d)$ be a given initial datum. Let $(v_0^\e)_\e \subset L^\infty(\T^d)$ be \emph{any} sequence of functions such that $v_0^\e \to u_0$ strongly in $L^1(\T^d)$ and let $v^\e$ and $u$ be the unique solutions of \eqref{eq:VV_2} and \eqref{eq:transport} with initial datum $v_0^\e$ and $u_0$ respectively. Then, there exist $\bar{\e}$ and a continuous function $\phi_{u_0}:\R^+\to\R^+$ with $\phi_{u_0}(0)=0$, such that for any $1\leq q<\infty$
	\begin{equation}\label{rate-modulo}
		\sup_{t\in(0,T)}\|v_\e(t,\cdot)-u(t,\cdot)\|_{L^q}\leq C(T,p,q,\|u_0\|_{L^\infty},\|\nabla\b\|_{L^1L^p})\left( \delta(\e)+\frac{1}{|\ln \delta(\e)|} +\phi_{u_0}(\delta(\e))\right)^{1/q},
	\end{equation} 
	for any $\e\leq \bar{\e}$, where 
	\begin{equation}\label{def:delta}
		\delta(\e):=\max\{\sqrt{\e},\|v_0^\e-u_0 \|_{L^1}\}.
	\end{equation}
\end{prop}
\begin{proof}
	We show the estimate \eqref{rate-modulo} in the case $q=1$, the general case will follow by a straightforward interpolation of the spaces $L^1(\T^d)$ and $L^\infty(\T^d)$. Since $u_0\in L^\infty(\T^d)\subset L^1(\T^d)$, using the continuity of translations in $L^1$ we can infer that there exists a continuous function $\phi_{u_0}$ as in the statement of the theorem and $h_0>0$ such that
	\begin{equation*}
		\|u_0(\cdot+h)-u_0\|_{L^1}\leq \phi_{u_0}(h),\hspace{0.5cm}\mbox{for all }h\leq h_0.
	\end{equation*}
	Then, for any $\delta\leq h_0$, we can compute
	\begin{equation*}
		\begin{aligned}
			\|v_\e(t,\cdot)-u(t,\cdot)\|_{L^1}&=\|\mathbb{E}[v_0^\e(\Xe_{t,0})]-u_0(\X_{t,0})\|_{L^1}\\
			&\leq \|\mathbb{E}[v_0^\e(\Xe_{t,0})]-u_0(\Xe_{t,0})\|_{L^1}\\
			&+\iint_{A_\delta(t,0)} |u_0(\Xe_{t,0}(x,\omega))-u_0(\X_{t,0}(x))|\de\mathbb{P}(\omega)\de x\\
			&+\iint_{(\T^d\times\Omega)\setminus A_\delta(t,0)} |v_0^\e(\Xe_{t,0}(x,\omega))-u_0(\X_{t,0}(x))|\de\mathbb{P}(\omega)\de x\\
			&\leq \|v_0^\e-u_0 \|_{L^1}+2\|u_0\|_{L^\infty}\mathscr{L}^d\otimes\mathbb{P}(A_\delta(t,0))+\phi_{u_0}(\delta)\\
			&\leq \|v_0^\e-u_0 \|_{L^1}+2C(T,p)\|u_0\|_{L^\infty} \left( \frac{\e}{\delta^2|\ln\delta|}+\frac{\|\nabla \b\|_{L^1L^p}}{|\ln\delta|}\right)+\phi_{u_0}(\delta),
		\end{aligned}
	\end{equation*}
	where the set $A_\delta$ is defined as in \eqref{def:Adelta} and in the last line we have used the estimate in Lemma \ref{lem:stability-flows}. The proof follows by choosing $\delta(\e)$ as in \eqref{def:delta} and $\delta(\bar{\e})=h_0$.
\end{proof}
It is clear that the rate provided by Proposition \ref{thm:rate} is not completely explicit for two reasons: on the one hand, the convergence of the initial datum depends upon the choice of the approximation $v^\e_0$; on the other hand, the function $\phi_{u_0}$ is implicitly related to the regularity of the initial datum. For the former issue, since we deal with bounded initial datum, existence and uniqueness of solutions of \eqref{eq:VV_2} and \eqref{eq:transport} are guaranteed by \cite[Proposition 5.3]{LBL} and \cite{DPL}, thus we do not need the approximating sequence $v_0^\e$. Concerning the latter issue, the function $\phi_{u_0}$ can be explicitly constructed once the regularity of $u_0$ is known. Motivated by the results in \cite{BN}, we provide the following example.
\begin{cor}\label{cor:rate}
	Let $u_0\in H^1(\T^d)$. Assume that the hypothesis of Theorem \ref{thm:rate} hold with $v_0^\e=u_0$. Then,
	\begin{equation}\label{rate_bn}
		\sup_{t\in(0,T)}\|v_\e(t,\cdot)-u(t,\cdot)\|_{L^2}\leq \frac{C}{\sqrt{|\ln\e|}},
	\end{equation}
	where the constant $C>0$ depends on $T,p,\|u_0\|_{L^\infty},\|\nabla u_0\|_{L^2},\|\nabla \b\|_{L^1L^p}$.
\end{cor}
\begin{proof}
	It is enough to compute the function $\phi_{u_0}$. We have that
	\begin{align*}
		\|u_0(\cdot+h)-u_0\|_{L^2}\leq h \|\nabla u_0\|_{L^2},
	\end{align*}
	and then we conclude by applying Proposition \ref{thm:rate} with $\delta=\sqrt{\e}$ and $\phi_{u_0}(\delta)=\delta \|\nabla u_0\|_{L^2}$.
\end{proof}
It is interesting to compare the rate given by Corollary \ref{cor:rate} and the one in \cite[Theorem 3.3]{BN}. Under the same assumption on the initial datum, Corollary \ref{cor:rate} provides a rate of convergence for a more general class of vector fields, namely $\b\in L^1((0,T);W^{1,p}(\T^d))$ with $p>1$ instead of $\b\in L^\infty((0,T);W^{1,p}(\T^d))$ with $p>2$. On the other hand, we do not improve completely the rate in \cite{BN}: the rate in \eqref{rate_bn} is better if $2\leq p\leq 3$, while it is worst for $p>3$. We also observe that a key tool in \cite{BN} is a propagation-of-regularity result, which is not needed in our argument.\\

We finally show how with these techniques it is possible to give a quantitative stability estimate for advection-diffusion equations. We address this issue motivated by the recent results in \cite{S21}:
\begin{lem}\label{lem:stability-stoc-flows}
	Let $\b\in L^1((0,T);W^{1,p}(\T^d))$ be a divergence-free vector field, where $p>1$. Let $\X^{\e_1}_{t,s},\X^{\e_2}_{t,s}$ be the stochastic flows of $\b$ associated respectively to $\e_1,\e_2>0$. Then,
	\begin{equation}
		\int_{\T^d}\mathbb{E}[\mathsf{d}(\X^{\e_1}_{t,s}(x),\X^{\e_2}_{t,s}(x))]\de x\leq C(T,p)\left( \sqrt[4]{|\e_1-\e_2|}+\frac{\|\nabla \b\|_{L^1L^p}}{|\ln|\e_1-\e_2||} \right).
	\end{equation}
\end{lem}
\begin{proof}
	We just sketch the proof since it follows the same computations of Step 2 in Lemma \ref{lem:stability-flows}. Notice that the S.D.E. solved by the difference $\X^{\e_1}_{t,s}-\X^{\e_2}_{t,s}$ is
	\begin{equation*}
		\begin{cases}
			\de(\X^{\e_1}_{t,s}(x,\omega)-\X^{\e_2}_{t,s}(x,\omega))=(\b(s, \X^{\e_1}_{t,s}(x,\omega))-\b(s,\X^{\e_2}_{t,s}(x,\omega)))\,\de s+(\sqrt{2\e_1}-\sqrt{2\e_2})\de \bm W_{s}(\omega),\\
			\X^{\e_1}_{t,t}(x,\omega)-\X^{\e_2}_{t,t}(x,\omega)=0.
		\end{cases}
	\end{equation*}
	Then, by defining the function $q_{\delta}(y)=\ln\left(1+\frac{|y|^2}{\delta^2}\right)$ and the related $Q^\delta_{\e_1,\e_2}$ as
	\begin{equation*}
		Q^\delta_{\e_1,\e_2}(t,s,x,\omega):=q_\delta(\X^{\e_1}_{t,s}(x,\omega)-\X^{\e_2}_{t,s}(x,\omega))=\ln\left(1+\frac{|\X^{\e_1}_{t,s}(x,\omega)-\X^{\e_2}_{t,s}(x,\omega)|^2}{\delta^2}\right),
	\end{equation*}
	when we apply It\^{o}'s formula the contribution of the stochastic part is different, namely
	\begin{equation*}
		\int_{\T^d} \E\left[Q^\delta_{\e_1,\e_2}(t,s,x)\right]\de x\leq \frac{|\e_1-\e_2| (t-s)}{\delta^2}+C\int_s^t\int_{\T^d} \E\left[\frac{\left| \b(\tau,\X^{\e_1}_{t,\tau}(x))-\b(\tau,\X^{\e_2}_{t,\tau}(x)) \right|}{\delta+\left| \X^{\e_1}_{t,\tau}(x)-\X^{\e_2}_{t,\tau}(x)\right|}\right]\de x\de \tau.
	\end{equation*}
	The conclusion follows by defining the set $A_\delta$ as
	\begin{equation*}
		A_\delta(t,s):=\left\{(x,\omega)\in \T^d\times\Omega:\mathsf{d}(\X^{\e_1}_{t,s}(x,\omega),\X^{\e_2}_{t,s}(x,\omega))> \sqrt{\delta}\right\},
	\end{equation*}
	and doing the same computations as in Step 2 of Lemma \ref{lem:stability-flows}.
\end{proof}

Then, the estimate on the flows yields the following rate of convergence for solutions of \eqref{eq:VV_2}.
\begin{prop}\label{prop:rate-ad}
	Let $\b\in L^1((0,T);W^{1,p}(\T^d))$ be a divergence-free vector field with $p>1$ and $u_0\in L^\infty(\T^d)$. Let $\e_n$ be a sequence such that $\e_n\to\e>0$ and let $v_{\e_n},v_\e$ the unique solutions of \eqref{eq:VV_2} with initial datum $u_0$ and viscosity $\e_n,\e$ , respectively. Then, there exist $N(u_0,T)$ and a continuous function $\phi_{u_0}:\R^+\to\R^+$ with $\phi_{u_0}(0)=0$, such that 
	\begin{equation}\label{def:rate-stab-ad}
		\sup_{t\in(0,T)}\|v_{\e_n}(t,\cdot)-v_\e(t,\cdot)\|_{L^1}\leq C\left(\frac{1}{|\ln |\e_n-\e||} +\phi_{u_0}\left(\sqrt{|\e_n-\e|}\right)\right),
	\end{equation}
	for any $n\geq N(u_0,T)$, where the constant $C$ depends upon $T,p,\|u_0\|_{L^\infty},$ and $\|\nabla\b\|_{L^1L^p}$.
\end{prop}
\begin{proof}
	The proof follows by arguing exactly as in the one of Proposition \ref{prop:rate-ad} and using Lemma \ref{lem:stability-stoc-flows}.
\end{proof}

One can compare the rate given by Proposition \ref{prop:rate-ad} with the ones in \cite{LiLuo} and \cite{S21}. The rate in \eqref{def:rate-stab-ad} depends upon $\phi_{u_0}$ and cannot be better than $O\left( \frac{1}{|\ln|\e_n-\e||} \right)$, but provides convergence in the {\em strong} norm $C([0,T];L^1(\T^d))$. On the other hand, the rates of \cite{LiLuo} and \cite{S21} are of order $\sqrt{|\e_n-\e|}$ and $|\e_n-\e|$, respectively, but they are given for a logarithmic distance which instead metrizes {\em weak} convergence.

\appendix

\section{An Eulerian approach to the vanishing viscosity scheme}\label{appendix}

The goal of this appendix is to give a detailed proof of the following result, which is refinement of \cite[Theorem IV.1]{DPL}:

\begin{thm}\label{thm:main} Let $\b \in L^1((0,T); W^{1,1}(\T^d))$ be a divergence-free vector field and let $u_0 \in L^1(\T^d)$ be a given initial datum. Let $(v_0^\e)_\e \subset L^\infty(\T^d)$ be \emph{any} sequence of functions such that $v_0^\e \to u_0$ strongly in $L^1(\T^d)$. Then the sequence $(v_\e)_{\e>0} \subseteq L^\infty((0,T); L^\infty(\T^d)) \cap L^2((0,T); H^1(\T^d))$ of solutions to \eqref{eq:VV_2} converges in $C([0,T];L^1(\T^d))$ to 
	the (unique) Lagrangian solution to \eqref{eq:transport}.  
\end{thm}

We begin with the following simple remark: 

\begin{remark}[Equation with a forcing term]\label{rem:forcing}  The same conclusions of Proposition \ref{prop:parabolic_wp} apply as well to the equation with a forcing term. More precisely, if $\chi \in C^\infty((0,T) \times \T^d)$ is a smooth function and $v_0 \in L^\infty(\T^d)$, then the problem 
	\begin{equation*}
		\begin{cases}
			\partial_t v + \b \cdot \nabla v = \Delta v + \chi &  \text{ in } (0,T) \times \T^d \\
			v\vert_{t=0}=v_0 & \text{ in } \T^d 
		\end{cases}
	\end{equation*}
	has a unique solution $v \in L^\infty((0,T); L^2(\T^d)) \cap L^2((0,T),H^1(\T^d))$. Observe also that via the transformation $v(t,x) \mapsto v(T-t,- x)$ we deduce well-posedness results also for the backward equation 
	\begin{equation}\label{eq:AD_dual}
		\begin{cases}
			-\partial_t v - \b \cdot \nabla v = \Delta v + \chi &  \text{ in } (0,T) \times \T^d \\
			v\vert_{t=T}=v_T & \text{ in } \T^d.
		\end{cases}
	\end{equation}
	Notice that if $v_T = 0$ then the problem \eqref{eq:AD_dual} admits a unique solution in $L^\infty((0,T); L^\infty(\T^d))$ and that it holds  
	\begin{equation*}
		\| v \|_{L^\infty((0,T); L^\infty(\T^d))} \le C(\Vert \chi\Vert_{C^0(\T^d)}) < +\infty. 
	\end{equation*} 
\end{remark}

\begin{proof}[Proof of Theorem \ref{thm:main}]
	We split the proof in several steps. 
	
	{\bf \emph{Step 1. Parabolic well-posedness and compactness (equi-integrability).}} We begin with the study of the problem \eqref{eq:VV_2}. From Proposition \ref{prop:parabolic_wp}, we deduce that for every fixed $\e>0$ there exists a unique function $v_\e \in L^\infty((0,T); L^\infty(\T^d)) \cap L^2((0,T); H^1(\T^d))$ solving \eqref{eq:VV_2}, which moreover satisfies
	\begin{equation*}
		\| v_\e \|_{L^\infty((0,T); L^s(\T^d))} \le \| v_0^\e \|_{L^s(\T^d)}, 
	\end{equation*}
	for any $s \in [1,+\infty]$. Since $u_0 \in L^1$, the family $(v_\e)_{\e}$ is in general not equi-bounded neither in $L^\infty((0,T); L^\infty(\T^d))$ nor in $L^2((0,T);H^1(\R^d))$. However, since $v_0^\e \to u_0$ strongly in $L^1(\R^d)$, we get for $s=1$
	\begin{equation}\label{eq:bound_L^q}
		\| v_\e \|_{L^\infty((0,T); L^1(\T^d))} \le \| v_0^\e \|_{L^1(\T^d)} \le C < + \infty 
	\end{equation}
	for some constant $C>0$ independent of $\e$. This is, however, still not sufficient to obtain weak compactness in $L^1$, as we need to show the equi-integrability of the family $(v_\e)_{\e>0}$. To do so, we argue in the following way: since $v_0^\e \to u_0$ strongly in $L^1(\T^d)$, by Theorem \ref{thm:equiintegrability}, there exists a  convex, increasing function $\Psi \colon [0,+\infty] \to [0,+\infty]$ such that $\Psi(0)=0$ and 
	\begin{equation}\label{eq:const_def}
		\lim_{s \to \infty} \frac{\Psi(s)}{s} = \infty \qquad \text{and} \qquad \sup_{\e>0} \int_{\R^d} \Psi(|v_0^\e(x)|) \, \de x =: C < \infty. 
	\end{equation}
	Without loss of generality, we can assume that $\Psi$ is smooth. By an easy approximation argument (as already done several times before), we can multiply the equation \eqref{eq:VV_2} by $\Psi'(|v_\e|)$ and we obtain
	\begin{equation*}
		\frac{d}{dt} \int_{\T^d} \Psi(|v_\e(\tau,x)|) \, \de x  + \e \int_{\T^d} \Psi^{\prime\prime}(|v_\e(\tau,x)|) |\nabla(|v_\e|)|^2 \, \de x = 0.  
	\end{equation*}
	The convexity of $\Psi$ and an integration in time on $(0,t)$ give
	\begin{equation*}
		\int_{\T^d} \Psi(|v_\e(t,x)|) \, \de x  \le C,
	\end{equation*}
	where $C$ is the same constant as in \eqref{eq:const_def}. Since $t$ is arbitrary, 
	\begin{equation}\label{eq:bound_equi_integrability}
		\sup_{t \in (0,T)} \int_{\T^d} \Psi(|v_\e(t,x)|) \, \de x  \le C.
	\end{equation}
	Since the constant $C$ is independent of $\varepsilon$, we can resort to Point {\rm (iii)} of Theorem \ref{thm:equiintegrability} and we infer that the family $(v_\e)_{\e>0}$ is weakly-precompact in $L^\infty((0,T); L^1(\T^d))$. Therefore, there exists a function $u^{\mathsf {V}} \in L^\infty((0,T); L^1(\T^d))$ such that $v_\e \rightharpoonup u^{\mathsf V}$ in $L^\infty((0,T); L^1(\T^d))$ (up to a non-relabelled subsequence). 
	
	{ \bf \emph{Step 2. Identification of the limit via duality I.}} 
	We now want to exploit a duality argument. Let $\chi \in C^\infty((0,T) \times \T^d)$ be arbitrary. By Remark \ref{rem:forcing}, for every $\e>0$, there exists a unique function $\vartheta_\e \in L^\infty((0,T); L^\infty(\T^d)) \cap L^2((0,T); H^1(\T^d))$ solving 
	\begin{equation}\label{eq:AD_forcing}
		\begin{cases}
			-\partial_t \vartheta_\e  - \b \cdot \nabla \vartheta_\e = \e \Delta \vartheta_\e + \chi &   \text{ in } (0,T) \times \T^d \\
			\vartheta_\e\vert_{t=T} = 0 & \text{ in } \T^d.
		\end{cases}
	\end{equation}
	The family $(\vartheta_\e)_{\e>0}$ is uniformly bounded in $L^\infty((0,T); L^\infty(\T^d))$
	so, up to a subsequence, the family $(\vartheta_\e)_{\e>0}$ converges in $C([0,T]; w^*-L^{\infty}(\T^d))$ to a function $\vartheta\in C([0,T]; w^*-L^{\infty}(\T^d))$ solving the backward, inhomogenous transport equation 
	\begin{equation}\label{eq:transport_buono}
		\begin{cases}
			-\partial_t \vartheta  - \b \cdot \nabla \vartheta =  \chi &  \text{ in } (0,T) \times \T^d \\
			\vartheta\vert_{t=T} = 0  & \text{ in } \T^d.
		\end{cases}
	\end{equation}
	By \cite{DPL}, the problem \eqref{eq:transport_buono} is well-posed in $L^\infty((0,T);L^\infty(\T^d))$ and thus $\vartheta$ coincides with  \emph{the} unique solution to \eqref{eq:transport_buono} which lies in  $C([0,T]; L^{\infty}(\T^d))$. In addition, this implies that the whole sequence $(\vartheta_\e)_{\e>0}$ converges to $\vartheta$ (in other words, the passage to a subsequence is not needed). For future use, observe that it also holds that 
	\begin{equation}\label{eq:initial_datum_weak_continous}
		\vartheta_{\e}(0, \cdot ) \overset{*}{\rightharpoonup} \vartheta(0,\cdot ) \qquad \text{ in $w^*-L^{\infty}(\T^d)$}
	\end{equation}
	and via a straightforward computation one also obtains the Duhamel representation formula 
	\begin{equation}\label{eq:duhamel}
		\vartheta(t,\bm X_{0,t}(x)) = \int_t^T \chi(s, \bm X_{0,s}(x)) \, \de s, \qquad \forall x \in \T^d, \, t \in [0,T].
	\end{equation}
	
	{ \bf \emph{Step 3. Identification of the limit via duality II.}}
	We now consider the regularised versions of problems \eqref{eq:VV_2} and \eqref{eq:AD_forcing}. Let $\rho$ be a non-negative, radially symmetric convolution kernel and set for $\e, \delta>0$
	\begin{equation*}
		v_{\e}^{\delta} := v_\e \ast \rho^\delta, \qquad  \vartheta_{\e}^{\delta} := \vartheta_\e \ast \rho^\delta.
	\end{equation*}
	The smooth functions $v^{\e,\delta}$ and $\vartheta^{\e,\delta}$ solve respectively the problems 
	\begin{equation}\label{eq:VV_con_commutatore}
		\begin{cases}
			\partial_t v_{\e}^{\delta}+ \b \cdot \nabla v_{\e}^{\delta} = r_v^{\e,\delta} + \e\Delta v_{\e}^{\delta} &  \text{ in } (0,T)  \times \T^d \\
			v_{\e}^{\delta}\vert_{t=0}=v_0^{\e,\delta} & \text{ in } \T^d
		\end{cases}
	\end{equation}
	and
	\begin{equation}\label{eq:VV_duale_forzante_con_commutatore}
		\begin{cases}
			-\partial_t \vartheta_{\e}^{\delta} - \b \cdot \nabla \vartheta_{\e}^{\delta} = r_\vartheta^{\e,\delta} + \e\Delta \vartheta_{\e}^{\delta} + \chi^\delta &  \text{ in } (0,T) \times \T^d \\
			\vartheta_{\e}^{\delta} \vert_{t=T}=0 & \text{ in } \T^d,
		\end{cases}
	\end{equation}
	where 
	\begin{equation*}
		\chi^\delta := \chi \ast \rho^\delta
	\end{equation*}
	and the commutators are defined as 
	\begin{equation*}
		r_v^{\e,\delta} :=\b \cdot \nabla v_{\e}^{\delta} - (\b\cdot \nabla v_\e) \ast \rho^\delta \qquad \text { and } \qquad 
		r_\vartheta^{\e,\delta} :=\b \cdot \nabla \vartheta_{\e}^{\delta} - (\b\cdot \nabla \vartheta_\e) \ast \rho^\delta. 
	\end{equation*} 
	
	Multiplying \eqref{eq:VV_con_commutatore} times $\vartheta_{\e}^{\delta}$, applying Fubini's Theorem and integrating by parts in time and space we obtain 
	\begin{equation*}
		\begin{split}
			0 & = \iint_{(0,T) \times \T^d} \left[ (\partial_t v_{\e}^{\delta})\vartheta_{\e}^{\delta}+ \b \cdot (\nabla v_{\e}^{\delta} )\vartheta_{\e}^{\delta} - r_v^{\e,\delta}\vartheta_{\e}^{\delta} - \e(\Delta v_\e)\vartheta_{\e}^{\delta} \right] \, \de t \de x \\
			& = \iint_{(0,T) \times \T^d} v_{\e}^{\delta} \left[ - \partial_t \vartheta_{\e}^{\delta} - \b \cdot \nabla \vartheta_{\e}^{\delta} - \e \Delta \vartheta_{\e}^{\delta} \right] \, \de t\, \de x  - \int_{\T^d} v^{\e,\delta}(0,x)\vartheta_{\e}^{\delta}(0,x)\, \de x\\ 
			& \qquad  -\iint_{(0,T) \times \T^d} r_v^{\e,\delta}\vartheta_{\e}^{\delta} \, \de t \de x\\
			& \hspace{-.25cm}\overset{\eqref{eq:VV_duale_forzante_con_commutatore}}{=}  \iint_{(0,T) \times \T^d} v_{\e}^{\delta} ( r_\vartheta^{\e,\delta} + \chi^\delta )\, \de t\, \de x - \int_{\T^d} v^{\e,\delta}(0,x)\vartheta_{\e}^{\delta}(0,x)\, \de x - \iint_{(0,T) \times \T^d} r_v^{\e,\delta}\vartheta_{\e}^{\delta} \, \de t \de x \\
			& = \iint_{(0,T) \times \T^d} v_{\e}^{\delta}  \chi^\delta\, \de t \de x + \iint_{(0,T) \times \T^d} (v_{\e}^{\delta} r_\vartheta^{\e,\delta}  - r_v^{\e,\delta}\vartheta_{\e}^{\delta}) \, \de t\de x - \int_{\T^d} v^{\e,\delta}(0,x)\vartheta_{\e}^{\delta}(0,x)\, \de x \\
			& =: \textrm{(I) + (II) + (III)}.
		\end{split}
	\end{equation*}
	Keeping $\e>0$ fixed, we now send $\delta \to 0$. The two commutators can be written in the form
	\begin{equation*}
		r_v^{\e, \delta}(t,x) = \int_{\T^d} v_\e(t,x+\delta y) \left[ \frac{\b(t,x+\delta y) - \b(t,x)}{\delta}\right] \cdot \nabla \rho(y) \,\de y
	\end{equation*}
	and  
	\begin{equation*}
		r_\vartheta^{\e, \delta}(t,x) = \int_{\T^d} \vartheta_\e(t,x+\delta y) \left[ \frac{\b(t,x+\delta y) - \b(t,x)}{\delta}\right] \cdot \nabla \rho(y) \,\de y. 
	\end{equation*}
	Since $v_\e,\vartheta_\e \in L^\infty((0,T); L^\infty(\T^d))$, arguing as in the proof of Proposition \ref{prop:parabolic_wp}, we can conclude that both $ r_v^{\e, \delta}$ and $r_\vartheta^{\e, \delta}$ converge to $0$ strongly in $L^1((0,T) \times \T^d)$ as $\delta \to 0$. This observation, combined with the uniform $L^\infty$ bounds on $v_{\e}^{\delta}$ and $\vartheta_{\e}^{\delta}$, shows that $\textrm{(II)} \to 0$ as $\delta \to 0$.  
	
	For the term $\textrm{(I)}$, instead, we can use the strong convergence of $v_{\e}^{\delta} \to v_\e$ and the uniform convergence of $\chi^\delta \to \chi$. Finally, for $\textrm{(III)}$, by standard results about convolutions 
	\begin{equation*}
		v^{\e,\delta}(0, \cdot ) \to v^{\e}_0 
	\end{equation*}
	strongly in $L^1(\T^d)$ as $\delta \to 0$; furthermore, we have 
	\begin{equation*}
		\vartheta_{\e}^{\delta}(0,\cdot) \to \vartheta_{\e}(0,\cdot)
	\end{equation*}
	weakly$^*$ in $L^{\infty}(\T^d)$ as $\delta \to 0$. Such convergence follows from a standard argument in the framework of evolutionary PDEs (see e.g. \cite[Lemma 3.7]{DeLellisNote}) which establishes, in particular, the weak continuity in time of the solutions to advection-diffusion or transport equations. 
	
	Thus, for any $\e>0$, it holds 
	\begin{equation*}
		\iint_{(0,T) \times \T^d} v_{\e}(t,x)  \chi(t,x)\, \de t \de x = \int_{\T^d} v_0^\e(x)\vartheta_{\e}(0,x)\, \de x .
	\end{equation*}
	We now send $\e \to 0$, getting
	\begin{equation}\label{eq:one}
		\iint_{(0,T)\times \T^d} u^{\mathsf V} (t,x)\chi(t,x)\, \de t \de x = \int_{\T^d} u_0(x)\vartheta(0,x)\, \de x.  
	\end{equation}
	In the last passage, we have used:
	\begin{itemize}
		\item $v^\e \rightharpoonup u^{\mathsf V}$ weakly in $L^\infty ((0,T); L^1(\T^d))$; 
		\item $v^\e_0 \to u_0$ strongly in $L^1(\T^d)$; 
		\item  $\vartheta_{\e}(0,\cdot) \rightharpoonup \vartheta(0,\cdot)$ weakly$^*$ in $L^{\infty}(\T^d)$ by \eqref{eq:initial_datum_weak_continous}.
	\end{itemize}
	
	{ \bf \emph{Step 4. Duality of the Lagrangian solution.}}
	A direct computation shows that the Lagrangian solution $u^{\mathsf L}$ satisfies 
	\begin{equation}\label{eq:two}
		\begin{split}
			\iint_{(0,T) \times \T^d} u^{\mathsf L}(t,x) \chi(t,x)\, \de t \de x & = \iint_{(0,T) \times \T^d} u_0(\bm X_{t,0}(x)) \chi(t,x)\, \de t\de x \\
			& =\int_0^T \int_{\T^d} u_0(y) \chi(t,\bm X_{0,t}(y))\, \de t \, \de y \\
			& =\int_{\T^d} u_0(y)  \int_0^T  \chi(t,\bm X_{0,t}(y))\, \de t \, \de y \\
			& \hspace{-.15cm} \overset{\eqref{eq:duhamel}}{=}  \int_{\T^d} u_0(x)\vartheta(0,x)\, \de x .
		\end{split}
	\end{equation}
	Hence, comparing \eqref{eq:one} and \eqref{eq:two}, we obtain 
	\begin{equation*}
		\iint_{(0,T) \times \T^d} (u^{\mathsf V}(t,x) - u^{\mathsf L}(t,x)) \chi(t,x)\, \de t\de x = 0. 
	\end{equation*}
	Being $\chi\in C^\infty((0,T) \times \T^d)$ arbitrary, we have thus obtained $u^{\mathsf V} = u^{\mathsf L}$ a.e. and this concludes the proof.
	
	{ \bf \emph{Step 5. Upgrade to strong convergence.}} 
	The convergence of $(v^\e)_{\e>0}$ to the Lagrangian solution is strong in $C([0,T]; L^1(\T^d))$. This follows from \cite[Lemma 3.3]{CCS4}: indeed, the regularity assumption (H1') of \cite[Lemma 3.3]{CCS4} includes the case $\b\in L^1((0,T);W^{1,1}(\T^d))$, and the growth assumption (H2) is trivially satisfied as already remarked in Section 3.
\end{proof}

We observe that Theorem \ref{thm:main}, together with the same arguments of Remark \ref{rmk:no_anom_diss},
implies that \emph{no anomalous dissipation is possible} for the vanishing viscosity limit in the case $\b \in L^1((0,T);W^{1,1}(\T^d))$ and $u \in L^\infty((0,T); L^2(\T^d))$, even though the solution lacks the integrability required for the DiPerna-Lions' theory to apply.

\end{document}